\documentclass[11pt]{amsart}

\usepackage{amssymb}
\usepackage{amsmath}
\usepackage{amsthm}

\theoremstyle{plain}

\newtheorem{theorem}{Theorem}[section]

\newtheorem{lemma}[theorem]{Lemma}
\newtheorem{proposition}[theorem]{Proposition}
\newtheorem{corollary}[theorem]{Corollary}

\newtheorem{fact}[theorem]{Fact}

\theoremstyle{definition}

\newtheorem{definition}[theorem]{Definition}
\newtheorem{remark}[theorem]{Remark}
\newtheorem{example}[theorem]{Example}

\newtheorem*{problem*}{Problem}

\newcommand{\dist}{\mathrm{dist}}
\newcommand{\C}{\mathrm{c}}

\newcommand{\SX}{S_{X^*}}

\newcommand{\N}{\mathbb{N}}

\newcommand{\m}{\mathrm{m}}

\newcommand{\inte}{\mathrm{int}\:}
\newcommand{\cconv}{\overline{\mathrm{conv}}\:}
\newcommand{\ccone}{\overline{\mathrm{cone}}\:}
\newcommand{\conv}{{\mathrm{conv}}\:}
\renewcommand{\span}{{\mathrm{span}}\:}

\renewcommand{\epsilon}{\varepsilon}
\renewcommand{\phi}{\varphi}

\begin{document}

\title{Stability of a convex
feasibility problem}

\author{Carlo Alberto De Bernardi
}
\address{Dipartimento di Discipline Matematiche, Finanza Matematica ed Econometria, Universit\`{a} Cattolica del Sacro Cuore, Via Necchi 9, 20123 Milano, Italy}

\email{carloalberto.debernardi@unicatt.it}

\author{Enrico Miglierina
}
\address{Dipartimento di Discipline Matematiche, Finanza Matematica ed Econometria, Universit\`{a} Cattolica del Sacro Cuore, Via Necchi 9, 20123 Milano, Italy}

\email{enrico.miglierina@unicatt.it}

\author{Elena Molho
}
\address{Dipartimento di Scienze economiche e Aziendali, Universit\`{a} degli Studi di Pavia, Via San Felice 5, 27100 Pavia, Italy}

\email{elena.molho@unipv.it}



 \thanks{
}



\maketitle

\begin{abstract}
The 2-sets convex feasibility problem aims at finding a point in
the intersection of two closed convex sets $A$ and $B$ in a normed
space $X$. %
%
More generally, we can consider the problem of finding (if
possible) two points in $A$ and $B$, respectively, which minimize
the distance between the sets.

In the present paper, we study some stability properties for the
convex feasibility problem: we consider two sequences of sets,
each of them converging, with respect to a suitable notion of set
convergence, respectively, to $A$ and $B$. Under appropriate
assumptions on the original problem, we ensure that the solutions
of the perturbed problems converge to a solution of the original
problem. We consider both the finite-dimensional and the
infinite-dimensional case. Moreover, we provide several examples
that point out the role of our assumptions in the obtained
results.
\end{abstract}

\section{Introduction}
The convex feasibility problem is the classical problem of finding
a point in the intersection of a finite collection of closed and
convex sets (see \cite[Section~4.5]{BorweinZhu} for the main
results on this subject). Many concrete problems in applications
can be formulated as a convex feasibility problem. As typical
examples, we mention solution of convex inequalities, partial
differential equations, minimization of convex nonsmooth
functions, medical imaging, computerized tomography and image
reconstruction. For some details and other applications see, e.g.,
\cite{BauschkeBorwein} and the references therein. Moreover, it is
worth to mention the recent annotated bibliography \cite{Censor},
about projection methods, containing several references to the
convex feasibility problem and its applications.

Many efforts have been devoted to the study of algorithmic
procedures to solve convex feasibility problems, both from a
theoretical and from a computational point of view  (see, e.g.,
\cite{BauschkeBorwein,BorweinSimsTam,Hundal,BCOMB} and the
references therein).

Often in concrete applications data are affected by some
uncertainties. Hence stability  of  solutions with respect to data
perturbations is a desirable property, also in view of the
development  of a computational approach to solve the convex
feasibility problem. Our paper is devoted to investigate some
stability properties of the  $2$-sets convex feasibility problem
by using set convergence notions. We will also consider the case
of a pair of closed and convex sets with empty intersection: in
this case a solution of the problem is a pair of minimal distance
elements  of the two sets.

In this paper, we investigate a sequence of perturbed convex
feasibility problems whose data are obtained by considering two
sequences of closed and convex sets $\left\lbrace A_n
\right\rbrace$ and $\left\lbrace B_n \right\rbrace$ converging
respectively to the sets $A$ and $B$. If the intersection of $A_n$
and $B_n$ is empty, we consider, as a solution of the $n$-th
perturbed problem, the pair of elements $a_n\in A_n$  and $b_n\in
B_n$ such that the distance between $A_n$ and $B_n$ is $\left\Vert
a_{n}-b_{n}\right\Vert$.

Our aim is to find some conditions that guarantee the convergence
of the solutions of the perturbed convex feasibility problems to a
solution of the original convex feasibility problem.

We obtain some stability results both in the finite-dimensional
and in the infinite-dimensional framework, using the
Kuratowski-Painlev\'{e} convergence notion in the
finite-dimensional case and the Attouch-Wets convergence in the
infinite-dimensional setting. Moreover, we give some examples
showing that the assumptions that we use to guarantee the
stability features of a given convex feasibility problem cannot be
avoided, both in the finite and in the infinite-dimensional case.

The paper is organized as follows. Section~\ref{Notations and
preliminaries} is devoted to definitions and preliminary results,
mainly concerning the various notions of set-convergence.
Section~\ref{Finite-dimensional} presents  a stability result for
the convex feasibility problem when $A$ and $B$ are contained in a
finite-dimensional normed vector space and the sequences of closed
and convex sets $\left\lbrace A_n \right\rbrace$ and $\left\lbrace
B_n \right\rbrace$ converge in the Kuratowski-Painlev\'{e} sense
respectively to $A$ and $B$. Section~\ref{Infinite-dimensional} is
devoted to study the stability properties of a convex feasibility
problem in an infinite-dimensional setting. Here, we use the
Attouch-Wets convergence, that is stronger than the
Kuratowski-Painlev\'{e} convergence, even if they coincide in the
finite-dimensional setting. Moreover, it is worth to be noticed
that we obtain results concerning both weak and norm convergence
of the solutions of perturbed problems to a solution of the
original problem. In order to obtain the norm convergence of a
sequence of solutions of perturbed problems, we assume that $A$
has nonempty interior and it is locally uniformly rotund (LUR) at
a given solution $a$. Hence, we use a geometrical notion that
strengthens the convexity assumption used to prove the weak
convergence result. Finally, in Section~\ref{sectionesempi}, we
provide some rather involved examples in $\ell_2$ that point out
the role of our assumptions even in a Hilbert space framework.

\section{Notations and preliminaries}\label{Notations and preliminaries}

Throughout all this paper, $X$ denotes a real normed space with
the topological dual $X^*$. We
denote by $B_X$ and $S_X$ the closed unit ball and the unit sphere of $X$, respectively. 
For $x,y\in X$, $[x,y]$ denotes the closed segment in $X$ with
endpoints $x$ and $y$, and $(x,y)=[x,y]\setminus\{x,y\}$ is the
corresponding ``open'' segment. For a subset $K$ of $X$,
$\alpha>0$, and a functional $x^*\in \SX$ bounded on $K$, let
$$S(x^*,\alpha,K)=\{x\in
K;\, x^* x\geq\sup x^*(K)-\alpha\}$$ be the closed slice of $K$
given by $\alpha$ and $x^*$.

For a subset $A$ of $X$, we denote by $\inte A$, $\conv(A)$ and
$\cconv(A)$ the interior, the convex hull and the closed convex
hull of $A$, respectively. Moreover,
$$\overline{\mathrm{cone}}(A)=\cconv{\bigl([0,\infty)\cdot
A\bigr)}$$ is the closed convex cone generated by the set $A$. We
denote by $$\textstyle \mathrm{diam}(A)=\sup_{x,y\in A}\|x-y\|,$$
the (possibly infinite) diameter of $A$. For $x\in X$, let
$$\dist(x,A) =\inf_{a\in A} \|a-x\|.$$ Moreover, given $A,B$
nonempty subset of $X$,  we denote by $\dist(A,B)$ the usual
``distance'' between $A$ and $B$, that is,
$$ \dist(A,B)=\inf_{a\in A} \dist(a,B).$$

\subsection*{Convergence of sets}

By $\C(X)$ we denote the family of all nonempty closed subsets of
$X$.

Let  $\{A_n\}$ be a sequence in $\C(X)$ and let us consider the
following sets:

$$\textstyle \mathrm{Li}\ A_n=\{x\in X;\, x=\lim_n x_n, x_n\in A_n\}$$
and

$$\textstyle \mathrm{Ls}\ A_n=\{x=\lim_k x_k\in X;\, x_k\in A_{n_k}, \{n_k\}\, \hbox{is a subsequence of the integers} \}.$$

\begin{definition}
\label{Kuratowski} Let $\{A_n\}$ be a sequence in $\C(X)$ and
$A\in\C(X)$.
\begin{enumerate}
\item $\{A_n\}$ converges to $A$ for the lower
Kuratowski-Painlev\'{e} convergence if{f} $A\subset\mathrm{Li}
A_n$. \item $\{A_n\}$ converges to $A$ for the upper
Kuratowski-Painlev\'{e} convergence if{f} $A\supset\mathrm{Ls}
A_n$. \end{enumerate} Moreover, we say that $\{A_n\}$ converges to
$A$ for the Kuratowski-Painlev\'{e} convergence
($A_{n}\stackrel{K}{\rightarrow}A$) if{f} $\{A_n\}$ converges to
$A$ for the upper and the lower Kuratowski-Painlev\'{e}
convergence.
\end{definition}


Now, let us introduce the (extended) Hausdorff metric $h$ on
$\C(X)$. For $A,B\in\C(X)$, we define the excess of $A$ over $B$
as

$$e(A,B) = \sup_{a\in A} d(a,B).$$

\noindent Moreover, if $A\neq\emptyset$ and $B=\emptyset$ we put
$e(A,B)=\infty$,  if $A=\emptyset$ we put $e(A,B)=0$. We define

$$h(A,B)=\max \bigl\{ e(A,B),e(B,A) \bigr\}.$$

\begin{definition} A sequence $\{A_j\}$ in $\C(X)$ is said to
Hausdorff converge to $A\in\C(X)$ if $$\textstyle \lim_j h(A_j,A)
= 0.$$
\end{definition}


Finally, we introduce the so called Attouch-Wets convergence (see,
e.g., \cite[Definition~8.2.13]{LUCC}), which can be seen as a
localization of the Hausdorff convergence.  If $N\in\N$ and
$A,B\in\C(X)$, define
\begin{eqnarray*}
e_N(A,C) &=& e(A\cap N B_X, C)\in[0,\infty),\\
h_N(A,C) &=& \max\{e_N(A,C), e_N(C,A)\}.
\end{eqnarray*}

\begin{definition} A sequence $\{A_j\}$ in $\C(X)$ is said to
Attouch-Wets converge to $A\in\C(X)$ if, for each $N\in\N$,
$$\textstyle \lim_j h_N(A_j,A)= 0.$$
\end{definition}



We recall that in the finite-dimensional case the Attouch-Wets
convergence and the Kuratowski-Painlev\'{e} convergence coincide
(see, e.g., \cite[Section 8.2]{LUCC}).

In the sequel, we use the following easy-to-prove fact. For the
convenience of the reader we provide a proof.

\begin{fact}\label{distanzanulla} Let $A$ and $B$
two closed and convex subsets of a normed space $X$ Let $\{A_n\}$
and $\{B_n\}$ be two sequences of closed convex sets such that
$A_n\rightarrow A$ and $B_n\rightarrow B$ for the lower
Kuratowski-Painlev\'{e} convergence. Then $$\textstyle \limsup_n
\dist(A_n,B_n)\leq \dist(A,B).$$ In particular, if $A\cap
B\neq\emptyset$ we have $\lim_n \dist(A_n,B_n)=0$.
\end{fact}

\begin{proof} Let $\epsilon>0$ and let $x\in A$ and $y\in B$ be
such that $\dist(A,B)\leq\|x-y\|+\epsilon$. Since $A_n\rightarrow
A$ and $B_n\rightarrow B$ for the lower Kuratowski-Painlev\'{e}
convergence, there exist two sequences $\{x_n\}$ and $\{y_n\}$
such that $x_n\to x$, $y_n\to y$ and, for each $n\in\N$, $x_n\in
A_n$, $y_n\in B_n$. In particular, it eventually holds
$\|x_n-x\|\leq \epsilon$ and $\|y_n-y\|\leq \epsilon$. Hence, the
following inequalities eventually hold:
$$\dist(A_n,B_n)\leq\|x_n-y_n\|\leq\dist(A,B)+3\epsilon.$$
By the arbitrariness of $\epsilon>0$, we have the thesis.
\end{proof}

\section{Convergence of minimal distance points of a pair of convex
sets: the finite-dimensional case}\label{Finite-dimensional}

In this section, we denote by $X$ a finite-dimensional normed
space.

\begin{definition} Let $A,B$ be nonempty
closed convex set in $X$. Let
$$\m(A,B)=\{a\in A;\, \dist(a,B)=\dist(A,B)\}.$$
\end{definition}

It is easy to see that $\m(A,B)$ is a closed convex set.

\begin{definition}
Let $C$ be a non empty closed convex set and $x\in C$. Let us
define
$$D(x)=\{d\in X;\, x+td\in C,\ \forall t>0\}.$$
\end{definition}

\begin{remark} By \cite[Proposition~2.1.5]{AT2003}, if $x,y\in C$ then
$D(x)=D(y)$. That is, the set $D(x)$ does not depend on $x\in C$.
We denote this set, called the {\em asymptotic cone} of $C$, by
$C_\infty$.
\end{remark}

We prove the following lemma that will be useful in the sequel (it
can be seen as a slight generalization of
\cite[Proposition~2.1.9]{AT2003}).

\begin{lemma}\label{lemma:conirecessione} Let $A$ and $B$ be nonempty
closed convex sets in $X$ such that $\m(A,B)$ is nonempty. Then
$$A_\infty\cap B_\infty=[\m(A,B)]_\infty.$$
\end{lemma}

\begin{proof} Let $a\in A$ and $b\in B$ be such that
$\|b-a\|=\dist(A,B)$.

Let us prove that $A_\infty\cap B_\infty\subset[\m(A,B)]_\infty$.
Let $d\in A_\infty\cap B_\infty=D(a)\cap D(b)$. Since, for each
$t>0$,
$$\|a-b\|=\|a+td-(b+td)\|=\dist(A,B),$$
we have that $a+td\in\m(A,B)$, whenever $t>0$. Hence $d\in
[\m(A,B)]_\infty$.

For the reverse inclusion, suppose that $a+td\in \m(A,B)$,
whenever $t>0$. Clearly, $d\in A_\infty$. Now, we prove that $d\in
B_\infty$. Let us fix $t>0$ and $n\in\N$, and let us observe that
$$\dist(b+ntd, B)\leq2\,\dist(A,B).$$
Hence, there exists $d_n\in B$ such that
$$\|b+ntd-d_n\|\leq2\,\dist(A,B).$$
Then,
$$\textstyle \|b+td-(b+\frac{d_n-b}n)\|\leq\frac2n\,\dist(A,B).$$
By the arbitrariness of $n\in\N$, since $b+\frac{d_n-b}n\in B$,
and since $B$ is closed, it holds that $b+td\in B$. By the
arbitrariness of $t>0$, the thesis is proved. \end{proof}

The following theorem is the main result of this section. It
proves that, under mild assumption, the 2-sets convex feasibility
problem has a considerable degree of stability.

\begin{theorem}
\label{thm:stability principale} Let $\{ A_{n}\} $ and $\{
B_{n}\}$ be two sequences of nonempty closed convex sets in $X$,
$A$ and $B$ two nonempty closed convex subsets of $X$ such that
$$A_{n}\to A\quad\text{and}\quad B_{n}\to B,$$
for the Kuratowski-Painlev\'{e} convergence. Suppose that
$\m(A,B)$ is a nonempty bounded set. Let $\left\{ a_{n}\right\} $
and $\left\{ b_{n}\right\}$ be  sequences such that $a_{n}\in
A_{n}$, $b_{n}\in B_{n}$ ($n\in\N$) and
\[
\text{dist}(A_{n},B_{n})=\left\Vert a_{n}-b_{n}\right\Vert.
\]
Then there exists a subsequence $\left\{ a_{n_{k}}\right\}$ such
that
\[
\lim_{k\rightarrow\infty}a_{n_{k}}=c\in \m(A\,,B).
\]
Moreover, if $\m(A,B)=\{a\}$ then $a_n\to a$.
\end{theorem}

\begin{proof} Let us prove the first part of the theorem. By
Fact~\ref{distanzanulla}, it holds
\begin{equation}\label{eq:distanza tende a zero}
\limsup_{n} \|a_n-b_n\|\leq\dist(A,B).
\end{equation}
We claim that $\{a_n\}$ and $\{b_n\}$ are bounded.

Suppose that this is not the case and let $a\in A$ and $b\in B$
such that $\|a-b\|=\dist(A,B)$. Without loss of generality, we can
suppose that $\|a_n\|,\|b_n\|\to\infty$. By the lower part of the
convergence of $\left\{ A_{n}\right\} $ there exists a sequence
$\left\{ a_{n}'\right\} $ such that $a_{n}'\in A_{n}$ and
$a_{n}'\rightarrow a$. Since $A_{n}$ is a convex set, for any
$\alpha\in\left[0,1\right]$ it holds:
\[
\alpha a_{n}+(1-\alpha)a_{n}'=a_{n}'+\alpha(a_{n}-a_{n}')\in
A_{n}.
\]

The sequence
\[
\left\{ \frac{a_{n}-a_{n}'}{\left\Vert a_{n}-a_{n}'\right\Vert
}\right\}
\]
has a subsequence converging to $d\neq0.$ There is no loss of
generality in assuming
\[
\lim_{n\rightarrow+\infty}\frac{a_{n}-a_{n}'}{\left\Vert
a_{n}-a_{n}'\right\Vert }=d.
\]
Therefore, it holds
\[
a+\beta
d=\lim_{n\rightarrow+\infty}\left(a_{n}'+\frac{\beta}{\left\Vert
a_{n}-a_{n}'\right\Vert }(a_{n}-a_{n}')\right).
\]

Since for every $\beta>0$ there exists $n_{2}(\beta)\in\N$ such
that $\frac{\beta}{\left\Vert a_{n}-a_{n}'\right\Vert
}\in\left[0,1\right],$ whenever $n>n_{2}(\beta)$, it holds
\[
a+\beta d\in A,
\]
for every $\beta>0.$ Hence, $d\in A_{\infty}.$

Analogously, we may prove that
\[
\lim_{n\rightarrow+\infty}\frac{b_{n}-b_{n}'}{\left\Vert
b_{n}-b_{n}'\right\Vert }=d'\in B_{\infty},
\]
where $\left\{ b_{n}'\right\} $ is a sequence such that $b_{n}'\in
B_{n}$ and $b_{n}'\rightarrow b$.

Let us observe that $\{a'_n\}$, $\{b'_n\}$ and $\{a_n-b_n\}$ are
bounded sequences in $X$. Since $\|a_n\|,\|b_n\|\to\infty$, we
have $\|a_n-a'_n\|\sim\|b_n-b'_n\|$ and hence
$$d=\lim_{n\rightarrow+\infty}\frac{a_{n}-a_{n}'}{\left\Vert
a_{n}-a_{n}'\right\Vert
}=\lim_{n\rightarrow+\infty}\frac{b_{n}-b_{n}'}{\left\Vert
b_{n}-b_{n}'\right\Vert }=d',
$$
Therefore we have
\[
0\neq d\in A_{\infty}\cap B_{\infty}.
\]
By Lemma~\ref{lemma:conirecessione}, we have
\[
A_{\infty}\cap B_{\infty}=[\m(A,B)]_{\infty}.
\]
Then $\m(A,B)$ is not a bounded set, a contradiction.

By the claim  and compactness, there exist two subsequences
$\left\{ a_{n_{k}}\right\} $ and $\left\{ b_{n_{k}}\right\} $,
respectively of $\left\{ a_{n}\right\} $ and of $\left\{
b_{n}\right\} $, such that
\[
\lim_{k\rightarrow+\infty}a_{n_{k}}=u,\quad\lim_{k\rightarrow+\infty}b_{n_{k}}=v,
\]
where $u\in A$ and $v\in B$. By Fact~\ref{distanzanulla},
$\|u-v\|=\dist(A,B)$ and the thesis is proved.

The second part of the theorem follows easily by the first part.
\end{proof}

\begin{remark}\label{remark:stabilityfunctions} The above theorem can be proved in an alternative way,
by using known results concerning stability theory for convex
optimization problem. However, we preferred to present a direct
and more geometrical proof. We give a sketch of the alternative
proof below. (See, e.g., \cite{LUCC} for definitions and main
results about convergence of functions and well-posed problems).

Let $f,f_n:X\times X\to(\-\infty,\infty]$ ($n\in\N$) the convex
lower semicontinuous functions defined as follows. For each
$(x_1,x_2)\in X\times X$ and $n\in\N$, put
$$
f(x_1,x_2)=\begin{cases}
\|x_1-x_2\|\ \ \ \ \ \  &\text{if $x_1\in A$ and $x_2\in B$;}\\
\infty\ \ \ \ \ \  &\text{otherwise;}
\end{cases}
$$
and
$$
f_n(x_1,x_2)=\begin{cases}
\|x_1-x_2\|\ \ \ \ \ \  &\text{if $x_1\in A_n$ and $x_2\in B_n$;}\\
\infty\ \ \ \ \ \  &\text{otherwise.}
\end{cases}
$$
Since $A_{n}\to A$ and $B_{n}\to B$ for the
Kuratowski-Painlev\'{e} convergence (equivalently, for the
Attouch-Wets convergence), we have that $f_n\to f$ for the
Kuratowski-Painlev\'{e} convergence. Moreover, proceeding as in
the proof of Theorem~\ref{thm:stability principale}, we may prove
that $f$ is Tykhonov well-posed in the generalized sense. Hence,
we can apply \cite[Theorem~10.2.24]{LUCC} to obtain the thesis.
\end{remark}

Whenever the two limit sets are such that $A\cap B\neq\emptyset$,
we have the following corollary.

\begin{corollary}
\label{thm:stabilitycorollario} Let $\{ A_{n}\} $ and $\{ B_{n}\}$
be two sequences of nonempty closed convex sets in $X$, $A$ and
$B$ two nonempty closed convex subsets of $X$ such that
$$A_{n}\to A\quad\text{and}\quad B_{n}\to B,$$
for the Kuratowski-Painlev\'{e} convergence. Suppose that $A\cap
B$ is a nonempty bounded set. Let $\left\{ a_{n}\right\} $ and
$\left\{ b_{n}\right\}$ be  sequences such that $a_{n}\in A_{n}$,
$b_{n}\in B_{n}$ ($n\in\N$) and
\[
\text{dist}(A_{n},B_{n})=\left\Vert a_{n}-b_{n}\right\Vert.
\]
Then there exist two subsequences $\left\{ a_{n_{k}}\right\} $ and
$\left\{ b_{n_{k}}\right\} $ such that
\[
\lim_{k\rightarrow\infty}a_{n_{k}}=\lim_{k\rightarrow\infty}b_{n_{k}}=c\in
A\cap B.
\]
Moreover, if $A\cap B=\{c\}$ then $a_n,b_n\to c$.
\end{corollary}

The following examples show that both the assumptions in Theorem
\ref{thm:stability principale} play an independent role and each
of them cannot be deleted. The first one focuses on the role of
convexity assumptions.

\begin{example}
\label{exa:main theorem convex} Let us consider the sets
$(n\geq2)$:
\[
\begin{array}[t]{c}
\textstyle A_{n}=\left\{
\left(x_{1},x_{2}\right)\in\mathbb{R}^{2}:-1\leq
x_{1}\leq-\frac{1}{n},
\,{n}\frac{x_{1}+1}{n-1}\leq x_{2}\leq2+{n}\frac{x_{1}+1}{1-n}\right\} \cup\\
\textstyle \left\{
\left(x_{1},x_{2}\right)\in\mathbb{R}^{2}:-1\leq
x_{1}\leq-\frac{1}{2n},\,\frac{2}{x_{1}^{2}}\leq
x_{2}\leq8n^{2}\right\}
\end{array}
\]
and
\[
\begin{array}[t]{c}
B_{n}=\left\{ \left(x_{1},x_{2}\right)\in\mathbb{R}^{2}:\frac{1}{n}\leq x_{1}\leq1,\,{n}\frac{x_{1}-1}{n-1}\leq x_{2}\leq2+{n}\frac{x_{1}-1}{1-n}\right\} \cup\\
\left\{ \left(x_{1},x_{2}\right)\in\mathbb{R}^{2}:\frac{1}{2n}\leq
x_{1}\leq1,\,\frac{2}{x_{1}^{2}}\leq x_{2}\leq8n^{2}\right\} .
\end{array}
\]
The sequences $\left\{ A_{n}\right\} $ and $\left\{ B_{n}\right\}
$ converge respectively to
\[
\begin{array}[t]{c}
A=\left\{ \left(x_{1},x_{2}\right)\in\mathbb{R}^{2}:-1\leq x_{1}\leq0,\, x_{1}+1\leq x_{2}\leq1-x_{1}\right\} \cup\\
\left\{ \left(x_{1},x_{2}\right)\in\mathbb{R}^{2}:-1\leq
x_{1}<0,\, x_{2}\geq\frac{2}{x_{1}^{2}}\right\}
\end{array}
\]
and
\[
\begin{array}[t]{c}
B=\left\{ \left(x_{1},x_{2}\right)\in\mathbb{R}^{2}:0\leq x_{1}\leq1,\,1-x_{1}\leq x_{2}\leq1+x_{1}\right\} \cup\\
\left\{ \left(x_{1},x_{2}\right)\in\mathbb{R}^{2}:0\leq
x_{1}\leq1,\, x_{2}\geq\frac{2}{x_{1}^{2}}\right\} .
\end{array}
\]
It is easy to see that $A\cap B=\left\{ (0,1)\right\} $ and
\[
A_{n}\stackrel{K}{\rightarrow}A,\quad
B_{n}\stackrel{K}{\rightarrow B}.
\]

All the assumptions of Theorem \ref{thm:stability principale} are
satisfied except for the convexity of $A_{n}$ and $B_{n}$. The
minimal distance between the sets $A_{n}$ and $B_{n}$ is achieved
only at the pair of points
\[
a_{n}=\left(-\frac{1}{2n},8n^{2}\right)\in A_{n}\:\mathrm{and}\:
b_{n}=\left(\frac{1}{2n},8n^{2}\right)\in B_{n}.
\]
It is apparent that the sequences $\left\{ a_{n}\right\} $ and
$\left\{ b_{n}\right\} $ have no convergent subsequences. Hence
the thesis of Theorem \ref{thm:stability principale} does not
hold.
\end{example}

The second example proves that the boundedness assumption on the
set $\m(A,B)$ cannot be dropped.
\begin{example}
\label{exa:main theorem bounded assumption}Let $A_{n}$ and $B_{n}$
be defined as in Example~\ref{exa:main theorem convex}. Let us
consider the sets
\[
C_{n}=\mathrm{conv}(A_{n})\:\mathrm{and}\,
D_{n}=\mathrm{conv}(B_{n}).
\]
It is easy to see that
\[
C_{n}\stackrel{K}{\rightarrow}C,\quad
D_{n}\stackrel{K}{\rightarrow}D,
\]
where
\[
C=\left\{ \left(x_{1},x_{2}\right)\in\mathbb{R}^{2}:-1\leq
x_{1}\leq0,\, x_{1}+1\leq x_{2}\right\}
\]
and
\[
D=\left\{ \left(x_{1},x_{2}\right)\in\mathbb{R}^{2}:0\leq
x_{1}\leq1,\,1-x_{1}\leq x_{2}\right\} .
\]
Moreover, we have $C\cap D=\left\{
\left(x_{1},x_{2}\right)\in\mathbb{R}^{2}:x_{1}=0,\,
x_{2}\geq1\right\} $.

All the assumptions of Theorem~\ref{thm:stability principale} are
satisfied except for the boundedness of the set $C\cap D$. The
minimal distance between the sets $C_{n}$ and $D_{n}$ is achieved
only at the same pair of points $a_{n}\in C_{n}$ and $b_{n}\in
D_{n}$ as in Example~\ref{exa:main theorem convex}. Of course, as
in the previous example both the sequences $\left\{ a_{n}\right\}
$ and $\left\{ b_{n}\right\} $ have no convergent subsequences.
Therefore the thesis of Theorem~\ref{thm:stability principale}
does not hold.
\end{example}

\section{Convergence of minimal distance points of a pair of convex
sets: the infinite-dimensional case}\label{Infinite-dimensional}

In an infinite-dimensional setting, we need some strengthenings of
the assumptions to obtain stability results for our problems.
Indeed, Example~\ref{esempiocono-cono}, in
Section~\ref{sectionesempi}, shows that an analogue of
Theorem~\ref{thm:stability principale} does not hold, even if we
assume that the sequences of sets converge for the Hausdorff
convergence and that the space $X$ is a Hilbert space. In this
section, we prove that an additional geometric condition on the
limit sets ensures the stability result (see
Theorem~\ref{puntolur} below). Moreover, we use the Attouch-Wets
convergence of sets instead of the Kuratowski-Painlev\'{e}
convergence.

\medskip

We start with some definitions and preliminary results. Let us
recall that a {\em body} in $X$ is a closed convex set in $X$ with
nonempty interior.

\begin{definition}[{see, e.g., \cite[Definition~7.10]{FHHMZ}}] Let $A$ be a nonempty subset of a normed space $X$. A point $a\in A$
is called a strongly exposed point of $A$ if there exists  a
support functional $f\in X^*\setminus\{0\}$ for $A$ in $a$ (i.e.,
$f (a) = \sup f(A)$), such that  $x_n\to a$ for all sequences
$\{x_n\}$ in $A$ such that $\lim f(x_n) = \sup f(A)$. In this
case, we say that $f$ strongly exposes $A$ at $a$.
\end{definition}

Let us observe that $f\in S_{X^*}$ strongly exposes $A$ at $a$
if{f} $f(a)=\sup f(A)$ and
$$\mathrm{diam}\bigl(S(f,\alpha,A)\bigr)\to0 \text{ as } \alpha\to 0.$$

\begin{definition}
Let $A\subset X$ be a body. We say that $x\in\partial A$ is an
{\em LUR (locally uniformly rotund) point} of $A$ if for each
$\epsilon>0$ there exists $\delta>0$ such that if $y\in\partial A$
and $\dist(\partial A,(x+y)/2)<\delta$ then $\|x-y\|<\epsilon$. If
$A=B_X$, this definition coincides with the standard definition of
local uniform rotundity of the norm at $x$.

Moreover, we say that $A$ is an {\em LUR body} if each point in
$\partial A$ is an LUR point of $A$.
\end{definition}

\begin{lemma}\label{slicelimitatoselur} Let $A$ be a body in $X$
and suppose that $a\in\partial A$ is an LUR point of $A$. Then, if
$f\in S_{X^*}$ is a support functional for $A$ in $a$, $f$
strongly exposes $a$. Moreover, every slice $S$ of the form
$S=S(f,\alpha,A)$ is a bounded set.
\end{lemma}

The first part of the lemma is well-known in the case the body is
a ball (see e.g. \cite[Exercise~8.27]{FHHMZ}) and in the general
case the proof is similar. However, for the convenience of the
reader we include a proof.

\begin{proof} Without loss of generality, we can suppose that
$a=0$. Fix $w\in \inte A$ and observe that $f(w)<0$.

Let us prove the first part of the lemma. Let $\alpha>0$, $z\in
S=S(f,\alpha,A)$ and
$$\textstyle z'=z-\frac {f(z)}{f(w)}w.$$
Since $\frac z2\in A$ and $f(z')=0$, we have that $[\frac z2,\frac
{z'}2]\cap \partial A\neq\emptyset$. Hence
$$\textstyle \dist(\partial A,\frac z2)\leq \frac12\|z'-z\|\leq \frac12\frac{\|w\|}{|f(w)|}\alpha.$$
Since $a=0$ is an LUR point of $A$, if $\alpha\to0$ then
$\mathrm{diam}(S)\to0$ and the proof is concluded.

Now, the second part of the lemma follows easily. Suppose on the
contrary that there exists $\alpha>0$ such that $S=S(f,\alpha,A)$
is unbounded. Then there exists a sequence $\{y_n\}$ in
$S\setminus\{0\}$ such that $\|y_n\|\to\infty$. Put
$z_n=\frac{y_n}{\|y_n\|}$ and observe that $\|z_n\|=1$ and $z_n\in
S(f,\alpha/\|y_n\|,A)$, a contradiction by the first part of the
lemma.
\end{proof}

\begin{lemma}\label{convnorma} Let $X$ be a normed space. There exists a constant $\Gamma>0$ such
that if $R>1$, if $x,y,a,b\in X$ are such that $\|x\|,\|y\|< R$
and $\|a\|,\|b\|>2 R$, then, if $[x,a]\cap RS_X=\{a'\}$ and
$[y,b]\cap RS_X=\{b'\}$, it holds
$$\|b'-a'\|\leq \Gamma\max\{\|x\|,\|y\|,\|a-b\|\}.$$
\end{lemma}

\begin{proof} Let $\lambda,\mu\in(0,1)$ be such that $a'=\lambda
a+(1-\lambda)x$ and $b'=\mu b+(1-\mu)y$. By the triangle
inequality, it follows easily that
$$\textstyle \frac{R-\|y\|}{\|b\|-\|y\|}\leq\mu
.$$ Moreover, since

$$\textstyle R=\|\lambda
a+(1-\lambda)x\|\geq\lambda\|a\|-(1-\lambda)\|x\|,$$ we have
$$\textstyle \lambda\leq \frac{R+\|x\|}{\|a\|+\|x\|}.$$

Without loss of generality, we can assume that $\lambda>\mu$. If
we denote $$d=\max\{\|x\|,\|y\|,\|a-b\|\},$$ we have
\begin{eqnarray*}
\textstyle \|b'-a'\| &\leq& \textstyle \lambda
                   \|a-b\|+(1-\lambda)\|x-y\|+|\lambda-\mu|(\|y\|+\|b\|)\\
          &\leq& \textstyle
3d+\bigl(\frac{R+\|x\|}{\|a\|+\|x\|}-\frac{R-\|y\|}{\|b\|-\|y\|}\bigr)(\|y\|+\|b\|)\\
    &\leq& \textstyle
3d+\frac{R(\|b\|-\|a\|-\|y\|-\|x\|)+\|a\|\,\|y\|+\|b\|\,\|x\|}{(\|a\|+\|x\|)(\|b\|-\|y\|)}(\|y\|+\|b\|)\\
    &\leq& \textstyle
3d+\frac{R(\|b\|-\|a\|-\|y\|-\|x\|)+\|a\|\,\|y\|+\|b\|\,\|x\|}{(\|a\|+\|x\|)(\|b\|/2)}(2\|b\|)\\
    &\leq& \textstyle
3d+4\frac{R(\|b\|-\|a\|-\|y\|-\|x\|)}{2R}+4\frac{\|a\|\,\|y\|+\|b\|\,\|x\|}{\|a\|+\|x\|}\\
    &\leq& \textstyle
    3d+2\bigl(\bigl|\|b\|-\|a\|\bigr|\bigr)+\frac{4\|a\|\,\|y\|}{\|a\|+\|x\|}+\frac{4\|b\|\,\|x\|}{\|a\|+\|x\|}\\
&\leq& \textstyle
    5d+\frac{4\|a\|}{\|a\|+\|x\|}d+\frac{4\|b-a+a\|\,\|x\|}{\|a\|+\|x\|}\\
&\leq& \textstyle
    5d+4d+\frac{4\|x\|}{\|a\|+\|x\|}\|b-a\|+\frac{4\|a\|}{\|a\|+\|x\|}\|x\|\leq \textstyle 17d\\
    \end{eqnarray*}
 The proof is concluded if we set $\Gamma=17$

\end{proof}

%
%

The following theorem is the main result of this section.

\begin{theorem}\label{puntolur} Let $X$ be a normed space, $B$
a nonempty closed convex subsets of $X$, $A$ a body in $X$ and
$a\in\partial A$ an LUR point of $A$. Let $\{A_n\}$ and $\{B_n\}$
be two sequences of closed convex sets such that $A_n\rightarrow
A$ and $B_n\rightarrow B$ for the Attouch-Wets convergence.
Suppose that $\{a_n\}$ and $\{b_n\}$ are sequences in $X$ such
that $a_n\in A_n,\ b_n\in B_n$ ($n\in \N$) and
$$\dist(A_n,B_n)=\|a_n-b_n\|.$$
Suppose that at least one of the following conditions holds.
\begin{enumerate}
\item[(1)] $A\cap B=\{a\}$. \item[(2)] $A\cap B=\emptyset$ and
there exists $b\in B$ such that $\dist(A,B)=\|a-b\|$.
\end{enumerate}
Then $a_n\to a$ in the $\|\cdot\|$-topology.
\end{theorem}

\begin{proof} There is no loss of generality in assuming $a=0$.
Let us assume that (1) holds.

Since $\inte(A)\cap B=\emptyset$, by the Hahn-Banach theorem there
exists $f\in\SX$ such that $$\sup f(A)=0=\inf f(B).$$ In
particular, $f$ is a support functional for $A$ in $0$. Let
$\alpha>0$ and observe that, by Lemma~\ref{slicelimitatoselur},
there exists $r>1$ such that $S=S(f,3\alpha,A)\subset rB_X$. Put
$R=r+\alpha$.

We claim that $\{a_n\}$ and $\{b_n\}$ are eventually contained in
$2R B_X$. Suppose that this is not the case and let $\{a_{n_k}\}$
and $\{b_{n_k}\}$ be two subsequences such that $\|a_{n_k}\|>2R$
and $\|b_{n_k}\|>2R$ whenever $k\in\N$. Now, let $x_{n_k}\in
A_{n_k}$ and $y_{n_k}\in B_{n_k}$ be such that $\|x_{n_k}\|\to 0$
and $\|y_{n_k}\|\to 0$ as $k\to\infty$. Let $[x_{n_k},a_{n_k}]\cap
RS_X=\{a'_{n_k}\}$ and $[y_{n_k},b_{n_k}]\cap RS_X=\{b'_{n_k}\}$,
and observe that, by Lemma~\ref{convnorma}, it holds
$\|b'_{n_k}-a'_{n_k}\|\to0$ as $n\to\infty$.

Since $A_n\rightarrow A$ for the Attouch-Wets convergence,
$a'_{n_k}\in A_{n_k}\cap RS_X$ and
$$A= S(f,3\alpha,A)\cup[A\cap\{x\in X;\,f(x)\leq -3\alpha\}]\subset rB_X\cup\{x\in X;\,f(x)\leq -3\alpha\},$$
it eventually holds $a'_{n_k}\in\{x\in X;\,f(x)\leq -2\alpha\}$.

Analogously, since $B_n\rightarrow B$ for the Attouch-Wets
convergence, $b'_{n_k}\in B_{n_k}\cap RS_X$ and
$$B\subset \{x\in X;\,f(x)\geq 0\},$$
it eventually holds $b'_{n_k}\in\{x\in X;\,f(x)\geq -\alpha\}$.

In particular, it eventually holds $\|b'_{n_k}-a'_{n_k}\|\geq
f(b'_{n_k}-a'_{n_k})\geq \alpha$, a contradiction. Therefore our
claim is proved.

Now, since $\{a_n\}$ and $\{b_n\}$ are bounded, there exist
sequences $\{w_n\}\subset A$ and $\{z_n\}\subset B$ such that
$\|w_n-a_n\|\to0$ and $\|z_n-b_n\|\to0$. Since clearly
$\lim_n\|z_n-w_n\|=0$, it holds
$$0\leq\liminf_n[f(z_n)-\|w_n-z_n\|]\leq\liminf_n[f(w_n)]\leq\limsup_n f(w_n)\leq0,$$
and hence that $f(w_n)\to0$ as $n\to\infty$. Since, by
Lemma~\ref{slicelimitatoselur}, $f$ strongly exposes $0$, we have
that $w_n\to0$ and hence that $a_n\to0$ in the
$\|\cdot\|$-topology. This concludes the proof in case (1).
\medskip

If assumption (2) holds, the proof is similar, but some additional
efforts are needed. Let $d=\dist(A,B)$ and observe that:
\begin{enumerate}
\item $\inte(A)\cap(B+dB_X)=\emptyset$; \item $0\in B+dB_X$; \item
$\limsup_n \|a_n-b_n\|\leq d$
\end{enumerate}
 Then there exists $f\in\SX$ such
that $$\sup f(A)=0=\inf f(B+dB_X).$$ In particular, $f$ is a
support functional for $A$ in $0$ and $\inf f(B)=d$. Let $\Gamma$
be the constant given by Lemma~\ref{convnorma} and let us consider
$S=S(f,(\Gamma+2)d,A)$ and observe that, by
Lemma~\ref{slicelimitatoselur}, there exists $r>1$ such that
$S\subset rB_X$. Let $R=r+d$.

We claim that $\{a_n\}$ and $\{b_n\}$ are eventually contained in
$2R B_X$. Suppose that this is not the case and let $\{a_{n_k}\}$
and $\{b_{n_k}\}$ be two subsequences such that $\|a_{n_k}\|>2R$
and $\|b_{n_k}\|>2R$ whenever $k\in\N$. Now, let $x_{n_k}\in
A_{n_k}$ $y_{n_k}\in B_{n_k}$ be such that $x_{n_k}\to a$ and
$y_{n_k}\to b$ as $k\to\infty$. Let $[x_{n_k},a_{n_k}]\cap
RS_X=\{a'_{n_k}\}$ and $[y_{n_k},b_{n_k}]\cap RS_X=\{b'_{n_k}\}$,
and observe that, by Lemma~\ref{convnorma}, it eventually holds
$\|b'_{n_k}-a'_{n_k}\|< (\Gamma+1)d$.

Since $A_n\rightarrow A$ for the Attouch-Wets convergence,
$a'_{n_k}\in A_n\cap RS_X$ and

\begin{eqnarray*}
A&=& S(f,(\Gamma+2)d,A)\cup[A\cap\{x\in X;\,f(x)\leq -(\Gamma+2)d\}]\\
&\subset& rB_X\cup\{x\in X;\,f(x)\leq -(\Gamma+2)d\},
\end{eqnarray*}

it eventually holds $a'_{n_k}\in\{x\in X;\,f(x)\leq
-(\Gamma+1)d\}$.

Analogously, since $B_n\rightarrow B$ for the Attouch-Wets
convergence, $b'_{n_k}\in B_n\cap RS_X$ and
$$B\subset \{x\in X;\,f(x)\geq d\},$$
it eventually holds $b'_{n_k}\in\{x\in X;\,f(x)\geq 0\}$.

In particular, it eventually holds $\|b'_{n_k}-a'_{n_k}\|\geq
f(b'_{n_k}-a'_{n_k})\geq (\Gamma+1)d$, a contradiction and our
claim is proved.

Now, since $\{a_n\}$ and $\{b_n\}$ are bounded, there exist
sequences $\{w_n\}\subset A$ and $\{z_n\}\subset B$ such that
$\|w_n-a_n\|\to0$ and $\|z_n-b_n\|\to0$. Let us observe that
$$d\leq\liminf\|z_n-w_n\|\leq\limsup_n\|z_n-w_n\|=\limsup_n\|a_n-b_n\|\leq d$$ and
$$0\leq\liminf_n[f(z_n)-\|w_n-z_n\|]\leq\liminf_n[f(w_n)]\leq\limsup_n f(w_n)\leq0.$$
Hence, we obtain $f(w_n)\to0$ as $n\to\infty$. Since, by
Lemma~\ref{slicelimitatoselur}, $f$ strongly exposes $0$, we have
that $w_n\to0$ and hence that $a_n\to0$ in the
$\|\cdot\|$-topology. \end{proof}

\begin{remark} As in the finite-dimensional case (see Remark~\ref{remark:stabilityfunctions}), the theorem above can be proved in an alternative way,
by using known results concerning stability theory for convex
optimization problem. However, the well-posedness of the involved
problems requires a proof with techniques similar to those used in
Theorem~\ref{puntolur}. As in the finite-dimensional case, we
preferred to present a direct and more geometrical proof.
\end{remark}

If the limit sets $A$ and $B$ satisfy a strong condition about
non-separation, we obtain a result similar to
Corollary~\ref{thm:stabilitycorollario}.

\begin{proposition}\label{internononvuoto} Let $A$ and $B$
two closed convex subsets of a reflexive Banach space $X$ such
that $A\cap B$ is bounded and such that $(\inte A)\cap
B\neq\emptyset$. Let $\{A_n\}$ and $\{B_n\}$ be two sequences of
closed convex sets such that $A_n\rightarrow A$ and
$B_n\rightarrow B$ for the Attouch-Wets convergence. Suppose that
$\{a_n\}$ and $\{b_n\}$ are sequences in $X$ such that $a_n\in
A_n,\ b_n\in B_n$ ($n\in \N$) and
$$\dist(A_n,B_n)=\|a_n-b_n\|.$$
Then there exist two subsequences $\{a_{n_k}\}$ and $\{b_{n_k}\}$
that weakly converge to a point of $A\cap B$.
\end{proposition}

\begin{proof} Let us observe that, since $(\inte A)\cap
B\neq\emptyset$, the sets $A_n\cap B_n$ ($n\in\N$) are eventually
nonempty and hence $a_n$ and $b_n$ eventually coincide. Since $X$
is reflexive, it suffices to prove that $\{a_n\}$ and $\{b_n\}$
are bounded. By \cite[Corollary~9.2.8]{LUCC}, the sequence
$\{A_n\cap B_n\}$ converges to $A\cap B$ for the Attouch-Wets
convergence. Since $A\cap B $ is bounded, the thesis holds.
\end{proof}

By combining the above proposition with Theorem~\ref{puntolur}, we
obtain the following corollary.

\begin{corollary}\label{LUR}
Let $X$ be a reflexive Banach space $X$. Let $A$ be an LUR body of
$X$ and $B$ a closed convex subset of $X$ such that $A\cap B$ is
nonempty and bounded. Let $\{A_n\}$ and $\{B_n\}$ be two sequences
of closed convex sets such that $A_n\rightarrow A$ and
$B_n\rightarrow B$ for the Attouch-Wets convergence. Suppose that
$\{a_n\}$ and $\{b_n\}$ are sequences in $X$ such that $a_n\in
A_n,\ b_n\in B_n$ ($n\in \N$) and
$$\dist(A_n,B_n)=\|a_n-b_n\|.$$
Then there exist subsequences $\{a_{n_k}\}$ and $\{b_{n_k}\}$ that
weakly converge to a point $c\in A\cap B$. Moreover, if $(\inte
A)\cap B=\emptyset$ then $a_n, b_n\to c$ with respect to the norm
convergence.

\end{corollary}

\section{Examples and final remarks}\label{sectionesempi}

In this section we provide two examples to illustrate the role of
the assumptions in the infinite-dimensional case. We point out
that both of them are in $\ell_2$. Therefore, the assumptions used
in Section~\ref{Infinite-dimensional} cannot be avoided even in
the ``simplest'' infinite-dimensional space.

The following example shows that an analogous of
Theorem~\ref{thm:stability principale} does not hold in the
infinite-dimensional setting.

\begin{example}\label{esempiocono-semispazio} Let $X=\ell_2$ and $\{e_n\}_n$ its standard basis. Let $A,B,A_n,B_n\subset X$ ($n\in\N$, $n\geq2$)
be defined as follows.
\begin{eqnarray*}
A   &=& \textstyle \ccone(\{ e_k+\frac{1}{k} e_1;\, k\in\N\});\\
B   &=& \{x\in X;\, e_1^*(x)=   0\};\\
A_n   &=& \textstyle\cconv\bigl(\{\ln n\, e_n+\frac1n e_1\}\cup(\frac1n e_1+A)\bigr);\\
B_n   &=& B.
\end{eqnarray*}
Let $a_n=\ln n\, e_n+\frac1n e_1\in A_n$ and $b_n=\ln n\, e_n\in
B_n$. Then:
 \begin{enumerate}
 \item $A\cap B=\{0\}$; \item $A_n\to A$ and $B_n\to B$ for
the Hausdorff convergence (and, hence, for the Attouch-Wets
convergence); \item $\dist(A_n,B_n)=\|a_n-b_n\|$; \item
$\|a_n\|,\|b_n\|\to\infty$.
\end{enumerate}
\end{example}

\begin{proof} We just have to prove (i) and (ii), since the
proofs of (iii) and (iv) are straightforward.

 (i) For $n\in\N\setminus\{1\}$, let
$f_n=n e_1^*-e_n^*$ and $g_n=e_n^*$ and observe that $$\textstyle
\{ e_k+\frac{1}{k} e_1;\, k\in\N\}\subset\{x\in X;\, f_n(x)\geq
0,\, g_n(x)\geq 0\}.$$ Then $A\subset \bigcap_{n=1}^\infty\{x\in
X;\, f_n(x)\geq 0,\, g_n(x)\geq 0\}$. Now, if $x\in A\cap B$, it
holds $e_1^*(x)=0$, $f_n(x)=-e^*_n(x)\geq0$ and
$g_n(x)=e^*_n(x)\geq0$. Then $x=0$.

 (ii) We just have to prove that $A_n\to A$ for the Hausdorff
convergence. Let us observe that $$\textstyle \dist(a_n,\frac1n
e_1+A)\leq\|\frac1n e_1+\ln n(e_n+\frac1n e_1)-a_n\|=\frac{\ln
n}n.$$ Hence, it holds
$$\textstyle d_H(A_n,A)\leq d_H(A_n,\frac1n e_1+A)+d_H(\frac1n e_1+A,A)\leq \frac{\ln n}{n}+\frac1n,$$
and the proof is concluded. \end{proof}

\medskip

Given two sets $A,B\subset X$, we say that $A$ and $B$ are {\em
separated} if{f} there exists $x^*\in X^*\setminus\{0\}$ such that
$$\textstyle \sup x^*(A)\leq\inf x^*(B).$$
The following example shows that, in
Proposition~\ref{internononvuoto}, the condition $$(\inte A)\cap
B\neq\emptyset$$ cannot be replaced by the weaker condition ``$A$
and $B$ are not separated''.

\begin{example}\label{esempiocono-cono} Let us consider $X=\ell_2$ and for $n\in\N$ let us consider the following subsets of $X$.
\begin{eqnarray*}
C_n   &=& \textstyle \ccone\bigl(\{e_{2n-1},e_{2n}+\frac1n e_{2n-1}\}\bigr)-\frac1n e_{2n-1};\\
D_n   &=& \textstyle \ccone\bigl(\{-e_{2n-1},e_{2n}-\frac1n e_{2n-1}\}\bigr)+\frac1n e_{2n-1};\\
C'_n   &=& \textstyle C_n-\frac1{\ln n}{e_{2n-1}};\\
D'_n   &=& \textstyle D_n-\frac1{\ln n}{e_{2n-1}};\\
A   &=& \textstyle \cconv\bigl(\bigcup_{n\in\N}C_n\bigr);\\
B   &=& \textstyle \cconv\bigl(\bigcup_{n\in\N}D_n\bigr);\\
A_n   &=& \textstyle \cconv\bigl(\bigcup_{k\in\N\setminus\{n\}}C_k\cup C'_n\bigr);\\
B_n   &=& \textstyle
\cconv\bigl(\bigcup_{k\in\N\setminus\{n\}}D_k\cup D'_n\bigr).
\end{eqnarray*}
 Then:
 \begin{enumerate}
\item $A$ and $B$ are not separated; \item $A\cap B$ is bounded;
 \item $A_n\to A$ and $B_n\to B$ for
the Hausdorff convergence (and, hence, for the Attouch-Wets
convergence); \item let $a_n=b_n=ne_{2n}\in A_n\cap B_n$ then
$\dist(A_n,B_n)=\|a_n-b_n\|=0$ and $\|a_n\|=\|b_n\|=n\to\infty$.
\end{enumerate}
\end{example}

Let us define $X_n=\span\{e_{2n-1},e_{2n}\}$, $B_{X_n}=B_X\cap
X_n$ and $Y_N=\span(\bigcup_{n=1}^N X_n)$. Observe that
$$Y_N=X_1\oplus_2\ldots\oplus_2 X_N,$$ where we denote by $X_1\oplus_2\ldots\oplus_2
X_N$ the direct sum $X_1\oplus\ldots\oplus X_N$ endowed with the
norm
$$\textstyle \|(x_1,\ldots,x_N)\|=(\|x_1\|^2+\ldots+\|x_N\|^2)^{\frac12}.$$

To exploit the features of Example~\ref{esempiocono-cono} we need
some preliminary lemmas. The easy proof of the following lemma is
left to the reader.

\begin{lemma}\label{intersezione} Let $C_n, D_n\subset X_n$ be
defined as above, then the following inclusion holds.
$$\textstyle (C_n+\frac1{\sqrt{n^2+1}} B_{X_n})\cap(D_n+\frac1{\sqrt{n^2+1}} B_{X_n})\subset 2B_{X_n}$$
\end{lemma}

\begin{lemma}\label{bolla-L2-L1} Let $W_n$ be convex subsets of $X_n$ containing the origin ($n=1,\ldots,N$) and
let $\epsilon>0$, then the following inclusion holds:
$$\textstyle \conv(\bigcup_{n=1}^N W_n+\epsilon B_{Y_N})\subset 2\,\conv(\bigcup_{n=1}^N [W_n+\sqrt N\epsilon B_{X_n}]).$$
\end{lemma}

\begin{proof} Since $Y_N=X_1\oplus_2\ldots\oplus_2 X_N$, it is
not difficult to prove that $$B_{Y_N}\subset \sqrt
N\conv(\bigcup_{n=1}^N B_{X_n}),$$ hence the following inclusions
hold:

\begin{eqnarray*}
\textstyle \conv(\bigcup_{n=1}^N W_n+\epsilon B_{Y_N}) &\subset&
\textstyle \conv(\bigcup_{n=1}^N W_n+\epsilon \sqrt
N\conv(\bigcup_{n=1}^N
B_{X_n}))\\
   &\subset&2\,\conv(\bigcup_{n=1}^N [W_n+\sqrt N\epsilon
   B_{X_n}]).
\end{eqnarray*}

\end{proof}

\begin{lemma}\label{scritturaquasiunica} For $n=1,\ldots,N$, let $W_n$ and $Z_n$ be convex subsets of $X_n$ containing the origin.
Then the following inclusion holds:
$$\textstyle \conv\bigl(\bigcup_{n=1}^N W_n\bigr)\cap\conv\bigl(\bigcup_{n=1}^N Z_n\bigr)\subset 2\conv\bigl(\bigcup_{n=1}^N
W_n\cap Z_n\bigr)$$
\end{lemma}

\begin{proof} Let $x\in \conv\bigl(\bigcup_{n=1}^N
W_n\bigr)\cap\conv\bigl(\bigcup_{n=1}^N Z_n\bigr)$, then there
exist $\alpha_n, \beta_n\in[0,1]$, $w_n\in W_n$ and $z_n\in Z_n$
($n=1,\ldots,N$) such that
$$\textstyle x=\sum_{i=1}^N \alpha_n w_n=\sum_{i=1}^N \beta_n z_n.$$
Since $Y_N=X_1\oplus\ldots\oplus X_N$, it holds $\alpha_n
w_n=\beta_n z_n$ ($n=1,\ldots,N$). Now suppose that
$\alpha_n\geq\beta_n>0$, then $w_n=\frac{\beta_n}{\alpha_n}z_n\in
W_n\cap Z_n$. Analogously, if $0<\alpha_n\leq\beta_n$, then
$z_n=\frac{\alpha_n}{\beta_n}w_n\in W_n\cap Z_n$. Hence
$$x\in(\alpha_1+\beta_1)(W_1\cap Z_1)+\ldots+(\alpha_N+\beta_N)(W_N\cap Z_N)\subset2\conv\bigl[\bigcup_{n=1}^N
(W_n\cap Z_n)\bigr].$$ \end{proof}

\begin{proof}[Proof of Example~\ref{esempiocono-cono}.]

Let us prove assertions (i), (ii) and (iii); the proof of (iv) is
obvious.

 (i) Let us observe that, for each $n\in\N$, the segments
$[-\frac1n e_{2n-1},\frac1n e_{2n-1}]$ and $[0, e_{2n}]$ are
contained in $A\cap B$. Now, suppose that there exists $f\in X^*$
such that $\sup f(A)\leq\inf f(B)$, then $f$ is constant on $A\cap
B$ and, by the above remark, it holds $f(e_n)=0$ whenever
$n\in\N$. Hence $f=0$.

(ii) Let us prove that $A\cap B$ is bounded. For $k\in\N$, let us
denote by $P_k$ the canonical projection on the first $k$
coordinates. Let $x\in A\cap B$ and let $N\in\N$ be such that
$\|x-P_{2N}x\|\leq1$. To conclude the proof it suffices to show
that $\|P_{2N}x\|\leq8$

We claim that $P_{2N} x\in \cconv\bigl(\bigcup_{n=1}^N C_n\bigr)$.
Indeed, since $x\in A$, there exists a sequence $\{y_k\}$,
converging in norm to $x$, such that
$y_k\in\conv\bigl(\bigcup_{n=1}^k C_n\bigr)$. Then the sequence
$\{P_{2N}y_k\}\subset \cconv\bigl(\bigcup_{n=1}^N C_n\bigr)$
converges in norm to $P_{2N}x$ and the claim is proved.

Analogously, it holds $P_{2N} x\in \cconv\bigl(\bigcup_{n=1}^N
D_n\bigr)$ and hence,
\begin{eqnarray*}
P_{2N} x   &\in& \textstyle [\conv\bigl(\bigcup_{n=1}^N
C_n\bigr)+\frac1{\sqrt{N^3+N}}B_{Y_N}]\cap[\conv\bigl(\bigcup_{n=1}^N
D_n\bigr)+\frac1{\sqrt{N^3+N)}}B_{Y_N}]\\
   &{\subset}& \textstyle 2\, \conv\bigl(\bigcup_{n=1}^N
[C_n+\frac1{\sqrt{N^2+1}}B_{X_n}]\bigr)\cap2\,
\conv\bigl(\bigcup_{n=1}^N
[D_n+\frac1{\sqrt{N^2+1}}B_{X_n}]\bigr)\\
   &{\subset}& \textstyle 4\, \conv\bigl(\bigcup_{n=1}^N
[C_n+\frac1{\sqrt{N^2+1}}B_{X_n}]\cap
[D_n+\frac1{\sqrt{N^2+1}}B_{X_n}]\bigr)\\
   &{\subset}& \textstyle 4\, \conv\bigl(\bigcup_{n=1}^N
2B_{X_n}\bigr)\subset 8B_X, \end{eqnarray*}

\noindent where the above inclusions hold by
Lemma~\ref{bolla-L2-L1}, Lemma~\ref{scritturaquasiunica} and
Lemma~\ref{intersezione}, respectively.

(iii) Let us prove that $A_n\to A$ for the Hausdorff convergence,
the proof that $B_n\to B$ for the Hausdorff convergence is
similar. Let us observe that $d_H(C_n,C'_n)=\frac1{\ln n}$, hence
we have:
$$\textstyle d_H(A,A_n)=d_H(\cconv\bigl(\bigcup_{k\in\N\setminus\{n\}}C_k\cup C'_n\bigr),
\cconv\bigl(\bigcup_{k\in\N\setminus\{n\}}C_k\cup
C_n\bigr))\leq\frac1{\ln n},$$ and the proof is concluded.
\end{proof}

\section*{Acknowledgments.}
The research of the first and second authors is partially
supported by GNAMPA-INdAM. The research of the second and third
authors is partially supported by Ministerio de Econom\'{i}a y
Competitividad (Spain), MTM2015-68103-P, Plan Nacional de
Matem\'{a}ticas, (2016-2018).

%
%
%

\end{document}